\documentclass[preprint,12pt,authoryear]{elsarticle}

\usepackage{latexsym,amsmath,amsthm,amsfonts}
\usepackage[psamsfonts]{amssymb}
\usepackage{color}
\newtheorem{theorem}{\sc Theorem}
\newtheorem{corollary}[theorem]{\sc Corollary}
\newtheorem{lemma}[theorem]{\sc Lemma}
\newtheorem{proposition}[theorem]{\sc Proposition}

\newtheorem{assum}[theorem]{\sc Assumption}
\newtheorem{remark}[theorem]{\sc Remark}

\theoremstyle{definition}

\theoremstyle{remark}

\newcommand{\R}{{\rm I}\kern-0.18em{\rm R}}

\def\E{\mathbb{E}}
\def\PP{\mathbb{P}}
\def\QQ{\mathbb{Q}}

\def\argmin{\mathop{\rm arg\, min}}

\def\g{\gamma}

\def\f{\varphi}

\def\l{\lambda}

\def\s{\sigma}

\def\cN{\mathcal N}
\def\cK{\mathcal K}
\def\tr{\mathrm{tr}}

\journal{?}

\begin{document}

\begin{frontmatter}



\title{Estimation of Low-Rank Covariance Function}


\author[m1]{Koltchinskii, V.\footnote{vlad@math.gatech.edu}\footnote{Supported in part by NSF Grants DMS-1207808 and CCF-1415498}}
\author[m1]{Lounici, K.\footnote{klounici@math.gatech.edu}\footnote{Supported in part by NSF CAREER Grant DMS-1454515 and Simons Collaboration Grant 315477}}
\author[m2]{Tsybakov, A.B.\footnote{alexandre.tsybakov@ensae.fr}\footnote{Supported by GENES and by the French National Research Agency (ANR) under the grants 
IPANEMA (ANR-13-BSH1-0004-02) and Labex ECODEC (ANR - 11-LABEX-0047)}}
\address[m1]{Georgia Institute of Technology, 686 Cherry St, Atlanta GA 30332, USA}
\address[m2]{Laboratoire de Statistique, CREST-ENSAE, 3, av. P.Larousse, 92240 Malakoff, France.}
%\ead[m1]{vlad@math.gatech.edu}
%\ead[m1]{klounici@math.gatech.edu}
%\ead[m2]{alexandre.tsybakov@ensae.fr}

\begin{abstract}
We consider the problem of estimating a low rank covariance function $K(t,u)$ of a Gaussian process $S(t), t\in [0,1]$ based on $n$ i.i.d. copies of $S$ observed in a white noise. We suggest a new estimation procedure adapting simultaneously to the low rank structure and the smoothness of the covariance function. The new procedure is based on nuclear norm penalization and exhibits superior performances as compared to the sample covariance function by a polynomial factor in the sample size $n$. Other results include a minimax lower bound for estimation of low-rank covariance functions showing that our procedure is optimal as well as a scheme to estimate the unknown noise variance of the Gaussian process.
\end{abstract}


\end{frontmatter}


\section{Introduction}\label{intro}

Let $X(t), t\in [0,1]$ be a Gaussian process satisfying the following stochastic differential equation:
\begin{align}\label{model}
dX(t) = S(t)dt + \sigma dW(t),\quad t\in [0,1],
\end{align}
where $W$ is the standard Brownian motion, $\sigma>0$ is the noise level, and %$S$ is a random process of the form
$$
S(t)=\sum_{k= 1}^r \sqrt{\lambda_k}\xi_k \varphi_k(t),\quad  t\in [0,1].
$$
Here $\xi_k$ are i.i.d. standard Gaussian random variables independent of the Brownian motion $W,$  $\{\varphi_k\}_{k=1}^r$ are unknown orthonormal functions in $L_2[0,1],$ possibly, with $r=\infty$, and the coefficients $\lambda_k>0$ are unknown and such that $\sum_{k= 1}^r \lambda_k<\infty$.  The value of $r$ is also unknown. 

Assume that we observe $n$ i.i.d. copies $X_1(t),\dots, X_n(t)$ of the process $X(t)$. In this paper, we study the problem of estimation of the covariance function of the stochastic 
process $S(\cdot),$
\begin{align}
\label{K}
K(t,u) =\E(S(t)S(u))=\sum_{k= 1}^{r} \lambda_k\varphi_k(t) \varphi_k(u), \quad t,u\in [0,1],
\end{align}
based on the observations $\{X_1(t),\dots, X_n(t), t\in [0,1]\}$. If $r=\infty,$ the sum in \eqref{K} is understood in the sense of  $L_2([0,1]\times [0,1])$-convergence. 
In short, (\ref{model}) is a model of a ``signal'' (Gaussian stochastic process $S$) observed in a Gaussian white noise and the 
goal is to estimate the covariance of the signal based on a sample of such observations. 

Statistical estimation of covariance functions has already received some attention in the literature. However, somewhat different setting was considered where the trajectories $X_i(\cdot)$ are observed at discrete time locations:
$$
Y_{i,j} = S(T_{i,j}) + \sigma \xi_{i,j},\quad 1\leq i \leq n,\; 1\leq j \leq m,
$$
where $\xi_{i,j}$ are i.i.d. $\cN(0,1)$ and, for each $i$, the points $T_{i,j}$, $1\leq j \leq m$, are equispaced in the interval $[0,1]$ or independent random variables with uniform distribution on $[0,1]$. In this setting, \cite{YaoMullerWang2005} proposed a local smoothing estimation procedure assuming that the trajectories $X_i(\cdot)$ are well approximated by the projection on the linear span of functions $\varphi_1,\dots,\varphi_k$ for some known fixed $k$ chosen by cross-validation. 
 This procedure is computationally intensive as it requires to compute the eigenvalues and the inverse for $n$ distinct $m\times m$ empirical covariance matrices of the trajectories $X_i$, $1\leq i \leq n$, at each of the cross-validation steps. %(See (4)-(11) in \cite{YaoMullerWang2005} for more details). 
 The results in \cite{YaoMullerWang2005} provide theoretical guarantees  for estimation of the covariance function and its eigenfunctions  under the condition that the previous approximation is sufficiently precise. %, which remains an open question. 
 \cite{HallMullerWang2006} consider the same methodology and study the effect of the sampling rate on the estimation rate of the eigenfunctions. In a similar framework, \cite{BuneaXiao2013} propose a simpler procedure to estimate the eigenfunctions and obtain theoretical guarantees on the estimation error. %(See Section 5 there). 
 Their approach involves a dimension reduction step where the selection of the relevant eigenfunctions is performed by thresholding the eigenvalues of a correctly constructed empirical covariance matrix. In a similar setting, \cite{Bigot2010} consider the estimation of the covariance matrix of the process $S$ at sample points rather than that of the covariance function. This problem can be reduced to multivariate regression and  \cite{Bigot2010} develop a model selection approach to it resulting in some oracle inequalities.

Noteworthy, strong regularity conditions are usually imposed on the eigenfunctions $\f_k$ in the existing literature. In \cite{HallMullerWang2006} the eigenfunctions are assumed to %be twice continuously differentiable and to 
admit bounded derivatives of order at least two. In addition, the optimal bandwidth choice in the local smoothing approach used in \cite{HallMullerWang2006,YaoMullerWang2005} requires the knowledge of smoothness degree of the eigenfunctions. In \cite{BuneaXiao2013}, the eigenfunctions are assumed to be continuously differentiable with bounded derivatives, the sequence of eigenvalues belongs to a Sobolev ball with regularity $\beta>0$ and the optimal choice of the threshold in the dimension reduction step depends on $\beta$. %(see Assumption 2 and Theorem 4.3 there).

An interesting question is what are the optimal rates of estimation of the covariance function in a minimax sense. To our knowledge, it was not addressed in the literature.

In this paper, we assume that the trajectories $X_i(\cdot)$ are fully observed in time. Our aim is to understand the influence of the structure of the covariance function $K$ on the estimation rate. The main contributions of this paper are as follows:
\begin{enumerate}
\item We propose a simple data-driven procedure to estimate the covariance function and prove oracle inequalities for it based on recent results on high-dimensional matrix estimation.
\item We show that the proposed method is minimax optimal for estimation of $K$ in the $L_2$-norm whereas the empirical covariance estimator is suboptimal. %by a polynomial factor in the number of observations $n$.
\end{enumerate}

\section{Definitions and notations}\label{method}

Let $e_1(\cdot),e_2(\cdot),\ldots$ be an orthonormal basis of $L_2[0,1]$, which is assumed to be fixed throughout the paper. 
Denote by $\|\cdot\|_2$ the norms either of $L_2[0,1]$ or of  $L_2([0,1]\times [0,1])$ (according to the context) and by 
$\langle \cdot,\cdot \rangle$ the corresponding inner products. For any integer $l\geq 1$, %we define $L_l = \mathrm{l.s.}
%we denote by $P_l$ the orthogonal projection  onto the linear span of $L_l$. Set also %$\dot{X}^{(l)} = P_l(\dot{X})$, $S^{(l)} = P_l(S)$ and $\dot{W}^{(l)}=P_l(\dot{W})$. 
consider the orthogonal projection $S^{(l)}= \sum_{k=1}^l \langle e_k,S\rangle e_k$  of $S$ onto the linear span of $\left\lbrace e_1,\dots,e_l \right\rbrace$.  Set
\begin{align}\label{eq:proj}
\dot{X}^{(l)}= \sum_{k=1}^l \int_0^1 e_k(t)dX(t)\ e_k, \ \ 
\dot{W}^{(l)}=\sum_{k=1}^l \int_0^1 e_k(t)dW(t)\ e_k.
\end{align}

In view of (\ref{model}), we have
\begin{align*}
\dot{X}^{(l)} = S^{(l)} + \sigma \dot{W}^{(l)}.
\end{align*}
%with $S^{(l)} = \sum_{k=1}^l \langle e_k,S\rangle e_k.$ 
Similarly to \eqref{eq:proj}, we define the processes
\begin{align*}
\dot{X}_i^{(l)}= \sum_{k=1}^l \int_0^1 e_k(t)dX_i(t)\ e_k,  \quad i=1,\dots, n,
\end{align*}
and consider
 the empirical covariance function
\begin{align*}
R_n^{(l)}(t,u) = \frac{1}{n}\sum_{i=1}^n \dot{X}^{(l)}_i(t)\dot{X}_i^{(l)}(u),\quad t,u\in [0,1].
\end{align*}
Note that the expectation of $R_n^{(l)}(t,u)$ is
\begin{align*}
\E \left[R_n^{(l)}(t,u) \right] &= \E \left[S^{(l)}(t) S^{(l)}(u)\right] + \sigma^2 I^{(l)}(t,u)\\
&= K^{(l)}(t,u) + \sigma^2 I^{(l)}(t,u),
\end{align*}
with $ I^{(l)}(t,u) = \sum_{k=1}^l e_k(t)e_k(u)$ and 
$$
K^{(l)}(t,u) = \E\left[S^{(l)}(t) S^{(l)}(u)\right] = \sum_{m=1}^r \l_m\f_m^{(l)}(t) \f_m^{(l)}(u)
$$ 
where $ \f_m^{(l)} = \sum_{k=1}^l \langle e_k,\f_m\rangle e_k$ is the orthogonal projection of $\f_m$ onto the linear span of 
$\left\lbrace e_1,\dots,e_l \right\rbrace$. In what follows, we will consider the set of functions
\begin{align*}
{\mathcal S}_l = \biggl\{\sum_{j,k=1}^l s_{jk}(e_j\otimes e_k): \ s_{jk}=s_{kj}, \ j,k=1,\dots, l\biggr\}
\end{align*}
where $(e_j\otimes e_k)(t,s)=e_j(t)e_k(s).$
The set ${\mathcal S}_l$ consists of all symmetric kernels belonging to the linear 
span of $\{e_j \otimes e_k: j,k=1,\dots, l\}.$ 
Note that $K$ is not necessarily in $\mathcal S_l$ while $R_n^{(l)},K^{(l)},I^{(l)}\in \mathcal S_{l}$. It is easy to see that $K^{(l)}$ is the orthogonal projection of $K$ onto $\mathcal S_l$.

If no ambiguity is caused, for any $A\in \mathcal S_l,$ we will use the same symbol $A$ to denote the corresponding symmetric $l\times l$ matrix. 
For any function $A\in \mathcal S_l$ or any $l\times l$ matrix $A$ we denote by  $\|A\|_1$ and $\|A\|_\infty$ its nuclear and spectral norms, respectively. The trace and the rank of matrix $A$ are denoted by ${\rm tr}(A)$ and ${\rm rank}(A)$, and its Frobenius norm by $\|A\|_F$. Writing $A\ge 0$ for a matrix $A$ means that $A$ is non-negative definite.

%%%%%%%%%%%%

\section{Nuclear norm penalized estimator and its convergence rate}

In this section, we assume that the noise level $\sigma$ is known. For an integer $l\ge 1$, we define the estimator $\hat A^{(l)}$ of $K$ as a solution of the following penalized minimization problem
\begin{align}\label{Lasso}
\hat A^{(l)} \in \mathrm{argmin}_{A\in \mathcal S_l, A\geq 0} \left(\|R_n^{(l)}-A - \sigma^2  I^{(l)}\|_{2}^2 + \mu \|A\|_1 \right),
\end{align}
where $\mu>0$ is a regularization parameter to be tuned. Note that here we have $\|A\|_1 = \mathrm{tr}(A)$. The solution of \eqref{Lasso} is explicitly expressed via soft thesholding of the eigenvalues of the matrix $R_n^{(l)}- \sigma^2  I^{(l)}$ (cf. \cite{KoltLouTsy2011}).
The next theorem easily follows from the argument in the proof of Theorem 1 in \cite{KoltLouTsy2011} (see also \cite{lounici2014}).
\begin{theorem}\label{thm-main}
Let $n,l\geq 1$ be integers and let $X_1(\cdot),\dots,X_n(\cdot)$ be i.i.d. realizations of the process $X(\cdot)$ satisfying (\ref{model}). If $\mu \geq 2 \| R_n^{(l)}-K^{(l)} - \sigma^2  I^{(l)} \|_{\infty}$ then, for any $K$ satisfying \eqref{K} with $\sum_{k= 1}^r \lambda_k<\infty$ we have
\begin{align*}
\| \hat A^{(l)} - K \|_2^2 &\leq \inf_{A\in \mathcal S_l, A\geq 0} 
\left\lbrace \|A - K\|_2^2  + \min\left\lbrace  2\mu \|A\|_1, \frac{(1+\sqrt{2})^2}{8} \mu^2 \mathrm{rank}(A)  \right\rbrace  \right\rbrace.
\end{align*}
\end{theorem}
This theorem is a deterministic fact as soon as we have a proper bound on a single random variable, namely, the spectral norm $\| R_n^{(l)}-K^{(l)} - \sigma^2  I^{(l)} \|_{\infty}$. In other words, all stochastic effects in our problem are localized in the behaviour of this random variable and the choice of $\mu$ is driven by it as well. The next lemma provides a probabilistic bound on this random variable. 

\begin{lemma}\label{lem-sparse}
Let $n,l\geq 1$ be integers and let $X_1(\cdot),\dots,X_n(\cdot)$ be i.i.d. realizations of the process $X(\cdot)$ satisfying (\ref{model}).
Set
$
\lambda_{\max}=\sup_{1 \leq j \leq r}\lambda_j.
$
For any $t>0$ and $l\geq 1$, define 
\begin{align}\label{deltanlt}
\delta_n(l,t) = \max\left\lbrace \sqrt{\frac{l + t}{n}},\frac{l + t}{n} \right\rbrace.
\end{align}
Then, with probability at least $1-e^{-t}$, for any $K$ satisfying \eqref{K} with $\sum_{k= 1}^r \lambda_k<\infty$ we have
$$
\| R_n^{(l)}-K^{(l)} - \sigma^2  I^{(l)} \|_{\infty} \leq C(\lambda_{\max}+\sigma^2) \delta_n(l,t),
$$
for some absolute constant $C>0$.

\end{lemma}

\begin{proof}
Set ${\rm x}_i (l)= (\int_0^1 e_1(t)dX_i(t),\ldots,\int_0^1 e_l(t)dX_i(t) )^\top$ for any $1\leq i \leq n$ and $\hat B_{n,l} = \frac{1}{n} \sum_{i=1}^n {\rm x}_i (l){\rm x}_i(l)^\top$.  Note that ${\rm x}_i(l)$ are i.i.d. normal random vectors with mean 0 and covariance matrix  $B_l= K^{(l)} + \sigma^2  I^{(l)}$. Also $\|R_n^{(l)}-K^{(l)} - \sigma^2  I^{(l)}\|_{\infty}= \|\hat B_{n,l} - B_l\|_{\infty}$. Here, $I^{(l)}$ is the $l\times l$ identity matrix. 
Next, 
$$
\|\hat B_{n,l} - B_l\|_\infty \leq \|B_l\|_\infty \left\|\frac{1}{n}\sum_{i=1}^n Z_i Z_i^\top - I^{(l)}\right\|_\infty \leq (\lambda_{\max}+\sigma^2)\left\|\frac{1}{n}\sum_{i=1}^n Z_i Z_i^\top - I^{(l)}\right\|_\infty 
$$
where $Z_1,\ldots,Z_n$ are i.i.d. standard normal vectors in $\R^l$. Here we also used the fact that the following 
representation holds for random vectors ${\rm x}_i (l):$ ${\rm x}_i (l)=B_l^{1/2}Z_i.$ 
Applying Theorem 5.39 in \cite{Vershynin} to the random variable $\left \|\frac{1}{n}\sum_{i=1}^n Z_i Z_i^\top - I^{(l)}\right\|_\infty$ we get the result.
\end{proof}

Theorem \ref{thm-main} with Lemma \ref{lem-sparse} immediately imply the following result.
\begin{theorem}\label{cor-main}
Let $n,l\geq 1$ be integers and let $X_1(\cdot),\dots,X_n(\cdot)$ be i.i.d. realizations of the process $X(\cdot)$ satisfying (\ref{model}).
Take
$$
\mu = c (\lambda_{\max}+\sigma^2)\delta_{n}(l,t),
$$
for some sufficiently large absolute constant $c>0$. Define
$$
v_{n}(A,l,t)=\min\left\lbrace  (\lambda_{\max}+\sigma^2)\mathrm{tr}(A) \delta_{n}(l,t),   (\lambda_{\max}+\sigma^2) ^2 \mathrm{rank}(A)\delta_{n}^2(l,t) \right\rbrace.
$$
Let $t>0$. Then, with probability at least $1-e^{-t}$, for any $K$ satisfying \eqref{K} with $\sum_{k= 1}^r \lambda_k<\infty$ we have
\begin{align}\label{oracle}
\| \hat A^{(l)} - K \|_2^2 &\leq  \inf_{A\in \mathcal S_l, A\geq 0} 
\left\lbrace \|A - K\|_2^2  + C v_{n}(A,l,t) \right\rbrace
\end{align}
with some absolute constant $C>0$. 
\end{theorem}
The bound \eqref{oracle} is the main oracle inequality that we will use now to obtain minimax bounds on the risk of the estimator $ \hat A^{(l)}$. It is easy to check that 
$$
v_{n}(A,l,t)\leq (\lambda_{\max}+\sigma^2)^2 {\rm rank}(A)\frac{l+t}{n}.
$$
The above bound is trivial if $l+t\leq n.$ In the case $l+t>n,$ it follows from the bound 
$$
(\lambda_{\max}+\sigma^2){\rm tr}(A)\frac{l+t}{n} \leq  (\lambda_{\max}+\sigma^2)\lambda_{\max}{\rm rank}(A)\frac{l+t}{n}
\leq (\lambda_{\max}+\sigma^2)^2{\rm rank}(A)\frac{l+t}{n}.
$$
Combining Theorem \ref{cor-main} with the fact that, for a random variable $\eta$, $\E[|\eta|] = \int_{0}^\infty \mathbb P(|\eta|\geq t)dt$  and taking $A= K^{(l)}$, 
\begin{align}\label{borne lasso r-fini}
\E[\| \hat A^{(l)} - K \|_2^2 ] \leq \|K^{(l)} - K\|_2^2 + C (\lambda_{\max}+\sigma^2)^2 \frac{(r\wedge l)l}{n}
\end{align}
for some absolute constant $C>0$, where we have used that  $\mathrm{rank}(K^{(l)})\le r\wedge l$. This inequality is valid for all $K$ of the form \eqref{K}, with finite or infinite $r$. 

As a corollary, we get the following bound on the minimax risk over the class of covariance functions that admit a finite expansion with respect to the basis~$\{e_k\}$. Denote by ${\mathcal K}_{r,l}(\l_{\max})$ the class of all covariance functions  satisfying \eqref{K} such that
$K\in \mathcal S_l$ and  $\|K\|_\infty \le \l_{\max}$ where $\l_{\max}$ is a finite positive constant. Note that the system of functions $\{\f_k\}$ in this definition is not fixed and varies among all orthonormal systems in $L_2[0,1]$.
\begin{corollary}\label{cor:finite}
Under the assumptions of Theorem~\ref{cor-main}, we have
\begin{align*}\label{}
\sup_{K\in{\mathcal K}_{r,l}(\l_{\max}) }\E[\| \hat A^{(l)} - K \|_2^2 ] \leq C (\lambda_{\max}+\sigma^2)^2 \frac{(r\wedge l)l}{n}
\end{align*}
for some absolute constant $C>0$. 
\end{corollary}
It is interesting to compare the estimator $\hat A^{(l)}$ with the other natural estimator, which is the corrected empirical covariance function $$\bar A^{(l)} \triangleq R_n^{(l)} - \sigma^{2}I^{(l)}.$$
We have the following expression for the risk of $\bar A^{(l)}$.
\begin{proposition}\label{prop:empir_covar} For any $K$ satisfying \eqref{K} with $\sum_{k= 1}^r \lambda_k<\infty$ we have
\begin{align*}
\E[\| \bar A^{(l)} - K \|_2^2 ]  = \|K^{(l)} - K\|_2^2 +\frac{\|B_l\|_2^2 + [\tr(B_l)]^2}{n}
\end{align*}
where $B_l= K^{(l)} + \sigma^2  I^{(l)}$.
\end{proposition}
\begin{proof}
Set for brevity $B=B_l$, $ \hat B_n= \hat B_{n,l}$, ${\rm x}_i={\rm x}_i(l)$. Note that $\E(\bar A^{(l)}) = K^{(l)}$. 
The bias-variance decomposition of the risk of $\bar A^{(l)}$ yields
%On the other hand, using elementary properties of the Wishart distribution, we get that  satisfies
\begin{align*}\label{}
\E[\| \bar A^{(l)}  - K\|_2^2 ] & =  \|K^{(l)} - K\|_2^2 + \E[\|R_n^{(l)} - \E(R_n^{(l)})\|_2^2].%\\
%& =  \|K^{(l)} - K\|_2^2 + \Big\|\fracR_n^{(l)} - \E(R_n^{(l)})\Big\|_2^2
\end{align*}
Here, $\E[\|R_n^{(l)} - \E(R_n^{(l)})\|_2^2] =\E[\| \hat B_n - B\|_F^2] =   \E\big[\big\|\frac1n \sum_{i=1}^n W_i\big\|_F ^2\big]$ where $W_i ={\rm x}_i {\rm x}_i^\top - \E [{\rm x}_i {\rm x}_i^\top]$. Since the matrices $W_i$ are i.i.d. we find $ \E\big[\big\|\frac1n \sum_{i=1}^n W_i\big\|_F ^2\big]=
\E \,\tr \big(\frac1{n^2} \sum_{i,j=1}^n W_i^\top W_j \big) = \frac1n \tr\big(\E (W_1^\top W_1)\big)= \frac1n \big(\E (|{\rm x}_1|_2^4) - \tr(B^\top B)\big)$
where $|\cdot|_2$ denotes the Euclidean norm. Here, $\E (|{\rm x}_1|_2^4) - \tr(B^\top B) = \|B\|_2^2 + [\tr(B)]^2$ and the result follows.
\end{proof}
%%%%%
Since $\tr (B)\ge \s^2 l$, Proposition~\ref{prop:empir_covar} implies
\begin{align}\label{lower_empir0}
\E[\| \bar A^{(l)} - K \|_2^2 ]  &\ge \|K^{(l)} - K\|_2^2 + \frac{ \sigma^4 l^2}{n},\\
\inf_{K }\E[\| \bar A^{(l)} - K \|_2^2 ] &\geq \frac{ \sigma^4 l^2}{n}\label{lower_empir}
\end{align} 
where $\inf_{K }$ is the infimum over all $K$ satisfying \eqref{K} with $\sum_{k= 1}^r \lambda_k<\infty$. Comparing \eqref{lower_empir}
with Corollary~\ref{cor:finite} we see that the risk of the empirical estimator $\bar A^{(l)}$ on the class ${\mathcal K}_{r,l}$ is of the order greater than the risk of our estimator $\hat A^{(l)}$ when $r$ is smaller than $l$.

Our estimator also outperforms the estimator $\bar A^{(l)}$ for kernels $K$ that do not admit a finite expansion with respect to the basis $\{e_k\},$ but satisfy some regularity conditions. To this end, we introduce a specific norm that can be naturally interpreted as a version of the Sobolev norm for covariance functions. Fix the smoothness parameter $s>0$. For any symmetric function $K \,:\, [0,1]^2 \rightarrow \R$, we define
\begin{align*}
\|K\|_{s,2} &:= \|\Delta^s K\|_2 =  
\left( \sum_{k, k'\geq 1} k^{2s} \langle Ke_k, e_{k'} \rangle^2      \right)^{1/2},
\end{align*}
%\left(\mathrm{tr}(\Delta^s K K \Delta^s)\right)^{1/2}=
where $\Delta$ is an operator admitting the matrix representation $\mathrm{diag}(1,2,\cdots,k,\cdots)$ w.r.t the basis $(e_k)_{k\geq 1}$. Note that the norm $\|K\|_{s,2}$ depends on the basis $\{e_k\}$ but we do not indicate this dependence in the notation since $\{e_k\}$ is fixed.
Note also that if $K$ admits spectral representation (\ref{K}), then 
$$
\|K\|_{s,2}=\left(\mathrm{tr}(\Delta^{2s} K^2)\right)^{1/2} = \left(\sum_{k=1}^r \lambda_k^2 \|\varphi_k\|_{s,2}^2\right)^{1/2}, 
$$
where we use the notation 
$$
\|\varphi\|_{s,2}= \|\Delta^s \varphi\|  = \left( \sum_{k\geq 1} k^{2 s} \langle \varphi,e_k  \rangle^2 \right)^{1/2} 
$$
for a Sobolev type norm of a function $\varphi\,:\, [0,1] \rightarrow \R.$

\begin{assum}
\label{smooth_kernel}
Suppose the covariance function $K$ has finite rank $r$ and there exist constants $\l_{\max}>0$, $s>0$ and $\rho\geq 1$ such that $\|K\|_{\infty}\leq \lambda_{\max}$ and $\|K\|_{s,2}\leq \rho$.
\end{assum}

Denote by $\overline{\mathcal K}_r(s,\rho;\lambda_{\max})$ the class of all 
kernels $K$ satisfying Assumption \ref{smooth_kernel}.
\begin{theorem}
\label{th:smooth}
Given $r\geq 1, s>0, \rho>0$ and $\lambda_{\max}>0,$
set 
$$
\ell := \max\left(\left\lceil \left(\frac{\rho^2}{(\lambda_{\max}+\sigma^2)^2}\frac{n}{r}\right)^{1/(2s+1)}\right\rceil, 
\left\lceil \left(\frac{\rho^2 n}{(\lambda_{\max}+\sigma^2)^2}\right)^{1/(2s+2)}\right\rceil\right).
$$
Then, with some absolute constant $C>0,$ 
\begin{align}
\label{eq:th:smooth}
&\sup_{K\in \overline{\mathcal K}_r(s,\rho;\lambda_{\max})} \E[\| \hat A^{(\ell)} - K \|_2^2 ] \leq 
\\
& 
\nonumber
C  \min
\left((\lambda_{\max}+\sigma^2)^{4s/(2s+1)}\rho^{2/(2s+1)} \left(\frac{r}{n}\right)^{2s/(2s+1)}, 
(\lambda_{\max}+\sigma^2)^{2s/(s+1)}\rho^{2/(s+1)} n^{-s/(s+1)}\right). 
\end{align}
\end{theorem} 
 
\begin{proof}
Since $K$ satisfies Assumption \ref{smooth_kernel}, we have for any $l\geq 1$ that
$$
\|K - K^{(l)}\|_2^2  = 
\sum_{k\geq l+1} \sum_{k'=1}^{\infty}\langle Ke_k, e_{k'}\rangle^2 + \sum_{k'\geq l+1}\sum_{k=1}^l \langle Ke_k, e_{k'}\rangle^2 
$$
$$
\leq (l+1)^{-2s}\sum_{k\geq l+1} \sum_{k'=1}^{\infty}k^{2s}\langle Ke_k, e_{k'}\rangle^2 
+ (l+1)^{-2s}\sum_{k'\geq l+1} \sum_{k=1}^{\infty}(k')^{2s}\langle Ke_k, e_{k'}\rangle^2\leq 2\rho^2 l^{-2s}.
$$
Combining the previous display with (\ref{borne lasso r-fini}), we find that, for any $l\ge1$,
\begin{align*}\label{}
\E[\| \hat A^{(l)} - K \|_2^2 ] \leq \rho^2 l^{-2s} + C (\lambda_{\max}+\sigma^2)^2 \frac{(r\wedge l)l}{n}\,.
\end{align*}
The minimum of the right-hand side of this inequality is achieved for $l$ of the order of $\ell$. By setting $l=\ell,$ we obtain  \eqref{eq:th:smooth}. 
\end{proof}

Note that, if the rank $r$ is small, the problem of estimation of covariance function $K$ reduces to estimation 
of a small number $r$ of eigenfunctions and eigenvalues of $K.$
The rate in (\ref{eq:th:smooth}) is, in this case, of the order $O(n^{-2s/(2s+1)}),$
which coincides with a standard minimax error rate of estimation of a function of one 
variable of smoothness $s.$  On the other hand, when the rank $r$ is large (say, $r=+\infty$), 
the estimation error rate becomes $O(n^{-s/(s+1)}),$ which is the minimax rate of estimation 
of a function of two variables of smoothness $s.$ Similar error rates where studied earlier 
in matrix completion problems for smooth kernels on graphs (see \cite{Kolt}).

We consider now a class of kernels determined by the following assumption, which can be interpreted as a Sobolev type condition on the individual eigenfunctions $\varphi_j$. 

\begin{assum}\label{bias-cond}
The value $r$ is finite and there exist constants $s>0$, $c_*>0$ such that, for any $1\leq j \leq r$,
$
\|\varphi_j\|_{s,2} \leq c_*.
$
\end{assum}
Denote by $\cK_{r}(s,c_*;\l_{\max})$ the class of all kernels $K$ defined by \eqref{K} with eigenfunctions $\f_j$ satisfying Assumption~\ref{bias-cond} and such that $\|K\|_\infty < \l_{\max}$.  
\begin{theorem}\label{th:nonpar}
Let $l_1 =\max \big( \lceil n^{\frac{1}{2s+1}}\rceil, \, \lceil (rn)^{\frac{1}{2(s+1)}}\rceil \big)$, $n\ge 1$, $1\le r<\infty$. For any $s>0$, $c_*>0$, $\l_{\max}>0$ we have 
\begin{align}\label{eq:th:nonpar}
\sup_{K\in \cK_{r}(s,c_*;\l_{\max})} \E[\| \hat A^{(l_1)} - K \|_2^2 ] \leq C  \min\big(rn^{-\frac{2s}{2s+1}}, \, r^{\frac{1}{s+1}}n^{-\frac{s}{s+1}}\big)
\end{align}
where $C>0$ is a constant depending only on $\lambda_{\max}, \sigma$ and $c_*$.
\end{theorem}
\begin{proof} It is enough to observe that, for all $K\in \cK_{r}(s,c_*;\l_{\max}),$
$$
\|K\|_{s,2}^2 = \sum_{k=1}^r \lambda_k^2 \|\varphi_k\|_{s,2}^2 \leq c_*^2 \lambda_{\max}^2 r,
$$
implying that 
$\cK_{r}(s,c_*;\l_{\max})\subset \overline{\mathcal K}_r(s,\rho;\lambda_{\max})$
with $\rho =c_{\ast}\lambda_{\max}\sqrt{r}.$ Bound \eqref{eq:th:nonpar} now follows 
from \eqref{eq:th:smooth}.
\end{proof}

When $r$ is a fixed constant and $n$ is large, the rate in \eqref{eq:th:nonpar} is $O(n^{-\frac{2s}{2s+1}})$. The next theorem shows that this rate cannot be achieved by the corrected empirical covariance estimator $ \bar A^{(l)}$ whatever is the choice of $l$.

\begin{theorem}\label{th:nonpar:empir_covar}
Let  $n\ge 1$, $1\le r<\infty$. There exists $c_*>0$ such that for any $s>0$,  $\l_{\max}>0$ we have 
\begin{align}\label{eq:th:nonpar:empir_covar}
\inf_{l\ge1} \sup_{K\in \cK_{r}(s,c_*;\l_{\max})} \E[\| \bar A^{(l)} - K \|_2^2 ] \geq C  n^{-\frac{s}{s+1}}
\end{align}
where $C>0$ is a constant that can depend only on $\lambda_{\max}, \sigma$, $s$ and $c_*$.
\end{theorem}

\begin{proof}
Fix $l\geq 1$ and consider the function
$$
\f_1(t) = C_1\left(\sum_{k=1}^l \frac{e_k(t)}{k^{s+1}} + \sum_{k=l+1}^{2l} \frac{e_k(t)}{k^{s+1/2}}\right), \quad t\in[0,1],
$$
%where the sum is understood in the sense of  $L_2[0,1]$-convergence and 
where $C_1$ is a normalizing constant, depending only on $s,$ such that $\|\f_1\|_2=1$. 
By an easy computation, $\|\f_1\|_{s,2}\leq c'$ for a constant $c'$ depending only on $s.$ 

Set $\bar K(t,u)=\l_{\max}\f_1(t)\f_1(u)$. Then $\bar K\in \cK_{r}(s,c_*;\l_{\max})$ with $c_*=c'$.
Due to \eqref{lower_empir0},
\begin{align}\nonumber
\sup_{K\in \cK_{r}(s,c_*;\l_{\max})}  \E[\| \bar A^{(l)} - K \|_2^2 ]  &\ge \sup_{K\in \cK_{r}(s,c_*;\l_{\max})}  \|K^{(l)} - K\|_2^2 + \frac{ \sigma^4 l^2}{n}\\
&\ge \|\bar K^{(l)} - \bar K\|_2^2 + \frac{ \sigma^4 l^2}{n}.\label{eq:th:nonpar:empir_covar1}
\end{align}
Observe that 
$$
\varphi_1\otimes \varphi_1=\varphi_1^{(l)}\otimes \varphi_1^{(l)}+ (\varphi_1-\varphi_1^{(l)})\otimes \varphi_1^{(l)}+
\varphi_1\otimes (\varphi_1-\varphi_j^{(l)}).
$$
Therefore, 
\begin{align*}
& \| \f_1\otimes \f_1 - \varphi_1^{(l)}\otimes \varphi_1^{(l)} \|_2^2 
= \|(\varphi_1-\varphi_1^{(l)})\otimes \varphi_1^{(l)}\|_2^2+ \|\varphi_1\otimes (\varphi_1-\varphi_1^{(l)})\|_2^2
\\
&\qquad \qquad \geq \|\varphi_1\|_2^2\|\varphi_1-\varphi_1^{(l)}\|_2^2
=\|\varphi_1 - \varphi_1^{(l)}\|_2^2.
\end{align*}
This implies that 
$$
\|\bar K^{(l)}-\bar K \|_{2}^2\ge \l_{\max}^2\|\varphi_{1} - \varphi_{1}^{(l)}\|_2^2\ge  c\l_{\max}^2  l^{-2s}
$$
for some constant $c>0$ depending only on $s.$

Using this inequality in \eqref{eq:th:nonpar:empir_covar1} and taking the minimum over $l\ge 1,$ we obtain the result.
\end{proof}

%%%%%%%%%%%%%%%%%
%

\section{Adaptive Estimation}\label{adap}

We observe that the optimal choice of the parameter $l$ in theorems \ref{th:smooth} and \ref{th:nonpar} depends on the unknown parameters 
$\rho,$ $s$ and $r$ that quantify respectively the smoothness of the eigenfunctions of $K$ and their number.  In this section, we propose an adaptive estimator, which does not depend on $s$ and $r$ that attains the same rate as in Theorem \ref{th:smooth} or in Theorem \ref{th:nonpar}.  

First, we describe a general method of aggregating estimators.
Assume without loss of generality that the sample size $n$ is even. We split the sample of $n$ trajectories $\mathbb X =\{X_1,\ldots, X_n\}$ into two parts of equal size $n/2$, denoted $\mathbb X_1 = \{X_1,\ldots, X_{n/2}\}$ and $\mathbb X_2= \{X_{n/2+1},\ldots, X_n\}$. Fix an integer $L$.  %such that $1\leq L\leq n$. 
Using the sample $\mathbb X_1$, we construct a family of estimators $A^{(1)},\ldots, A^{(L)}$ such that $A^{(l)}\in \mathcal S_l$, $1\leq l \leq L$. These can be, for example, the estimators $\hat A^{(1)},\ldots, \hat A^{(L)}$ defined in (\ref{Lasso}).

Consider the following adaptive selector of $l$:
\begin{align}\label{hatl}
\hat l = \argmin_{1\leq l \leq L}\{ \| A^{(l)} \|_2^2 - 2 \langle A^{(l)}, \tilde R_n^{(l)}-\s^2I^{(l)}\rangle  \},
\end{align}
where $\tilde R_n^{(l)}(t,u)= \frac{2}{n}\sum_{i=n/2+1}^n \dot{X}_i^{(l)}(t) \dot{X}_i^{(l)}(u)$ is the projected empirical covariance function associated to the second subsample $\mathbb X_{2}$. 

In the following theorem we assume that the first subsample is frozen, so we state the result for non-random functions $A^{(l)}\in \mathcal S_l$, $1\leq l \leq L$.

\begin{theorem} \label{adap-scheme}
Let $A^{(l)}$, $1\leq l \leq L$, be functions such that $A^{(l)}\in \mathcal S_l$.
 For any $t>0$, with probability at least $1-e^{-t}$ with respect to the subsample $\mathbb X_{2}$  we have
\begin{align*}
\| A^{(\hat l)} - K \|_2^2 &\leq 2 \min_{1\leq l \leq L}\| A^{(l)} - K  \|_2^2 
 +
C[\lambda_{\max}\vee \sigma^2]^2 \max\left\lbrace \frac{t+\log L}{n},\left(\frac{t+\log L}{n}\right)^2 \right\rbrace
\end{align*}
for all $K$ satisfying \eqref{K} with $\sum_{k= 1}^r \lambda_k<\infty$.
Here, $C>0$ is an absolute constant.
\end{theorem}

%%%%%%%%%%%

\begin{proof} Fix an arbitrary $\bar l \in \{1,\dots,L\}$.
Note that, by definition, $\{\mathcal S_l\}_{l\geq 1}$ is a nested sequence satisfying
$$
\mathcal S_{l+1} = \mathcal S_l \oplus \mathrm{l.s.}\left\lbrace   e_{j}\otimes e_{l+1} + e_{l+1} \otimes e_{j},\, 1\leq j \leq l \right\rbrace.
$$
Consequently, for any $1\leq l,l'\leq L$, we have
$\langle A^{(l)}, \tilde R_n^{(l)} \rangle = \langle A^{(l)}, \tilde R_n^{(l\vee l')} \rangle$. Similarly $\langle A^{(l)}, K  \rangle = \langle A^{(l)}, K^{(l)}  \rangle = \langle A^{(l)}, K^{(l\vee l')}  \rangle$.
Combining this observation with (\ref{hatl}), we get
\begin{align*}
&\|A^{(\hat l)} - K\|_2^2 - \|A^{(\bar l)} - K\|_2^2 \\
&\hspace{1.5cm}= \| A^{(\hat l)}\|_2^2 - 2\langle A^{(\hat l)},K^{(\hat l)} \rangle - [\|A^{(\bar l)}\|_2^2 - 2\langle A^{(\bar l)}, K^{(\bar l)}\rangle]\\
&\hspace{1.5cm}\leq \|\hat A^{(\hat l)}\|_2^2 - 2\langle A^{(\hat l)},\tilde R_n^{(\hat l)} - \sigma^2 I^{(\hat l)} \rangle - [\|A^{(\bar l)}\|_2^2 - 2\langle A^{(\bar l)}, \tilde R_n^{(\bar l)} - \sigma^2 I^{(\bar l)}\rangle]\\
&\hspace{6cm} + 2\langle A^{(\hat l) } - A^{(\bar l)}, \tilde R_n^{(\hat l\vee \bar l )} - K^{(\hat l\vee \bar l )} -\sigma^2 I^{(\hat l \vee \bar l)}\rangle\\
&\hspace{1.5cm}\leq 2\langle A^{(\hat l )} - A^{(\bar l)}, \tilde R_n^{(\hat l\vee \bar l )} - K^{(\hat l\vee \bar l )} - \sigma^2 I^{(\hat l \vee \bar l)}\rangle.
\end{align*}
Here, $
 K^{(\hat l\vee \bar l )} +\sigma^2 I^{(\hat l \vee \bar l)}= \E[ \tilde R_n^{(\hat l\vee \bar l )}] 
$. Setting for brevity $m=\hat l\vee \bar l$ we deduce from the previous display that
\begin{align*}
\|A^{(\hat l)} - K\|_2^2 - \|A^{(\bar l)} - K\|_2^2 &\le 2 U \|A^{(\hat l )} - A^{(\bar l)}\|_2 
\le \frac16  \|A^{(\hat l )} - A^{(\bar l)}\|_2^2 +6U^2
\end{align*}
where $U \triangleq \max_{l=1,\dots,L}
\langle  U_l, \tilde R_n^{(m)} - \E[ \tilde R_n^{(m)}]\rangle$ with $U_l= (A^{(\hat l )} - A^{(\bar l)})/\|A^{(\hat l )} - A^{(\bar l)}\|_2$ if $A^{(\hat l )} \ne A^{(\bar l)}$ and $U_l=0$ otherwise. It follows from the last display and the bound 
$$
\frac16  \|A^{(\hat l )} - A^{(\bar l)}\|_2^2 \leq \frac{1}{3}\|A^{(\hat l)}-K\|_2^2 + \frac{1}{3}\|A^{(\bar l)}-K\|_2^2
$$
that 
\begin{align}\label{U2}
\|A^{(\hat l)} - K\|_2^2 \le  2 \|A^{(\bar l)} - K\|_2^2  +9 U^2.
\end{align}
Since $\bar l$ is arbitrary, to complete the proof it suffices to bound the random variable $U$ in probability. We first obtain a bound for each of the variables $\zeta_l= \langle  U_l, \tilde R_n^{(m)} - \E[ \tilde R_n^{(m)}]\rangle$. Note that associating $U_l$ with the corresponding  $m\times m$ matrices that we will also denote by $U_l$, we can write $\zeta_l =\langle  U_l, \hat B - B \rangle$ where $\hat B = (2/n) \sum_{i=n/2+1}^n {\rm x}_i (m){\rm x}_i(m)^\top$,  $B=K^{(m)} + \sigma^2  I^{(m)} = \E[{\rm x}_i(m){\rm x}_i(m)^\top]$, and 
${\rm x}_i(m)$ are i.i.d. normal vectors with mean 0 and covariance matrix  $B$ (cf.  the proof of Lemma~\ref{lem-sparse}) and $\langle  \cdot, \cdot \rangle$ is the inner product of matrices. It follows that
\begin{align*}
\zeta_l&= \langle B^{1/2} U_l B^{1/2}, \,\frac{2}{n}\sum_{i=n/2+1}^n Z_i Z_i^\top - I^{(m)}\rangle\\
&= \tr\Big( \frac{2}{n}\sum_{i=n/2+1}^n  B^{1/2} U_l B^{1/2}Z_i Z_i^\top - B^{1/2} U_l B^{1/2}\Big)\\
&= \frac{2}{n}\sum_{i=n/2+1}^n  Z_i^\top DZ_i - \tr(D)
\end{align*}
where $Z_1,\ldots,Z_n$ are i.i.d. standard normal vectors in $\R^m$ and $D=B^{1/2} U_l B^{1/2}$. By the  Hanson-Wright inequality (see, e.g., \cite{RudelsonVershynin}) we have that for any $t>0$, with probability at least $1-e^{-t}$, 
\begin{align}\label{HW}
\left|\frac{2}{n}\sum_{i=n/2+1}^n  Z_i^\top DZ_i - \tr(D)\right| \le C\left(\frac{\|D\|_\infty t}{n} +  \|D\|_F\sqrt{\frac{t}{n}}\right)
\end{align}
where $C>0$ is an absolute constant. Since $\|U_l\|_2\le1$ when considering $U_l$ as a function (which is equivalent to $\|U_l\|_F\le1$ when considering $U_l$ as a matrix) and $\|B\|_\infty\le \l_{\max} +\s^2$ we have $\|D\|_\infty\le \|D\|_F\le \l_{\max} +\s^2$. Thus, with probability at least $1-e^{-t}$
\begin{align*}
|\zeta_l| \le C  (\lambda_{\max} \vee \sigma^2) \left( \sqrt{\frac{t}{n}} + \frac{t}{n}\right)
\end{align*}
where $C>0$ is an absolute constant. 
The union bound argument gives that, with probability at least $1-e^{-t}$, 
$$
U^2= \max_{l=1,\dots,L} \zeta_l^2 \le C (\lambda_{\max} \vee \sigma^2)^2 \left( \sqrt{\frac{t+ \log L}{n}} + \frac{t+ \log L}{n}\right)^2
$$
where $C>0$ is an absolute constant. Combining this with \eqref{U2} proves the theorem.
\end{proof}

%%%%%%%%%%%%

We now apply Theorem \ref{adap-scheme} to $A^{(l)}=\hat A^{(l)}$ where the estimators $\hat A^{(1)},\ldots, \hat A^{(L)}$ are defined in (\ref{Lasso}).
Combining Theorems \ref{cor-main}, \ref{adap-scheme} and the fact that, for a random variable $\eta$, $\E[|\eta|] = \int_{0}^{\infty} \mathbb P \left(|\eta|\geq t \right)dt$ we get the following result.

\begin{theorem}\label{th:model_sel}
Let each of the estimators $\hat A^{(l)}$ satisfy the conditions of Theorem \ref{cor-main}. Then 
$$
\E\left[\|\hat A^{(\hat l)} - K\|_2^2\right] \le C\min_{1\leq l \leq L} \inf_{A \in \mathcal S_l,\, A\geq 0} \left\lbrace \|A-K\|_2^2  +  v_n(A,l,l) \right\rbrace  +C[\lambda_{\max}\vee \sigma^2]^2 \frac{\log L}{n}
$$
for all $K$ satisfying \eqref{K} with $\sum_{k= 1}^r \lambda_k<\infty$.
Here, 
$C>0$ is an absolute constant. 
\end{theorem}

We now fix $L=n$. Using Theorem~\ref{th:model_sel}, Theorem~\ref{th:smooth}, Theorem~\ref{th:nonpar} and Corollary~\ref{cor:finite}, we obtain the following result.

\begin{theorem}\label{th:adapt}
Let each of the estimators $A^{(l)}=\hat A^{(l)}$ satisfy the conditions of Theorem \ref{cor-main}. Let $\hat A^{(\hat l)}$ be the aggregated estimator with $\hat l$ defined in (\ref{hatl}) with $L=n$. 

(i) For any $r\geq 1$, $c_*>0$ and $s>0$ such that $1\le r\leq n^{1+2s}$, we have
\begin{align*}
\sup_{K\in \cK_{r}(s,c_*;\l_{\max})} \E\|\hat A^{(\hat l)} - K\|_2^2 \leq C \min\big(rn^{-\frac{2s}{2s+1}}, \, r^{\frac{1}{s+1}}n^{-\frac{s}{s+1}}\big),
\end{align*}
where $C>0$ is a constant that can depend only on $\lambda_{\max},\sigma^2, c_*$, and $s$. 

(ii) For any $r\geq 1$, $\rho\geq 1$, $s>0$, $\l_{\max}>0$, $\sigma^2\geq 0$ such that $\rho^2 \leq (\l_{\max} + \sigma^2)^2  \min \big( r n^{2s} , n^{1+2s} \big)$,
we have
\begin{align*}
\sup_{K\in \overline{\cK}_{r}(s,\rho;\l_{\max})} \E\|\hat A^{(\hat l)} - K\|_2^2 \leq C \min\left( \left(\frac{r}{n}\right)^{2s/(2s+1)}, \,  n^{-s/(s+1)}\right),
\end{align*}
where $C>0$ is a constant that can depend only on $\lambda_{\max},\sigma^2,\rho$, and $s$. 

(iii) If $(r\wedge l)l\ge \log n$ and $l\le n$, then for any $\l_{\max}>0$,
\begin{align*}\label{}
\sup_{K\in{\mathcal K}_{r,l}(\l_{\max}) }\E[\| \hat A^{(\hat l)} - K \|_2^2 ] \leq C \frac{(r\wedge l)l}{n}
\end{align*}
where $C>0$ is a constant that can depend only on $\lambda_{\max}$ and $\sigma^2$. 
\end{theorem}
The conditions $r\leq n^{1+2s}$ and $\rho^2 \leq (\l_{\max} + \sigma^2)^2  \min \big( r n^{2s} , n^{1+2s} \big)$ are rather mild. Indeed, if $r$ and $\rho$ are fixed quantities, then these conditions are satisfied for $n$ large enough. Theorem~\ref{th:adapt} shows that the estimator $\hat A^{(\hat l)}$ is adaptive to the unknown parameters $r$ and $s$ on the scale of classes $\overline{\cK}_{r}(s,\rho;\l_{\max})$ and $ \cK_{r}(s,c_*;\l_{\max})$ that no price is paid in the rate as compared to the non-adaptive estimators of Theorems~\ref{th:smooth} and \ref{th:nonpar}. The same estimator is adaptive on the scale of classes ${\mathcal K}_{r,l}(\l_{\max})$, again with no price to be paid, for a wide range of values of $l$ and $r$.

\section{Estimation of $\sigma^2$}\label{noise-est}
We now tackle the estimation of the unknown variance $\sigma^2$. We use the simple idea that $\langle e_l, S\rangle$ becomes negligible for large $l$ when Assumption \ref{bias-cond} is satisfied. Therefore, we propose the following (biased) estimator of $\sigma^2$ based on an independent copy $X$ of the process (\ref{model}):
\begin{align}
\hat\sigma^2 = \frac{1}{M} \left\|\dot X^{(L+M)}-\dot X^{(L)}\right\|_2^2,\quad L=e^{n}, M\geq 1.
\end{align}

\begin{theorem}
Let $n,l\geq 1$ be integers and let $X_1(\cdot),\dots,X_n(\cdot)$ be i.i.d. realizations of the process $X(\cdot)$ satisfying (\ref{model}). Let Assumption \ref{bias-cond} be satisfied. For any $t>0$, we have with probability at least $1-e^{-t}$
$$
|\hat\sigma^2 - \sigma^2| \lesssim \max\left\lbrace c_{*}^2r \lambda_{\max} L^{-2s} ( 1\vee \sqrt{t}\vee t) , \sigma^2 \sqrt{\frac{t}{M}} , \frac{t}{M}\right\rbrace.
$$

\end{theorem}

\begin{proof}

We have, in view of Plancherel inequality, that
\begin{align}\label{sigma-est}
&\hat\sigma^2  - \sigma^2= \frac{1}{M} \left\| S^{(L+M)}  - S^{(L)} \right\|_2^2 + \frac{2}{M}\langle S^{(L+M)}  - S^{(L)},\dot W^{(L+M)}-\dot W^{(L)}\rangle\notag\\
&\hspace{6cm}+ \frac{1}{M}\left\|\dot W^{(L+M)}-\dot W^{(L)}\right\|_2^2- \sigma^2\notag\\
&=  \frac{1}{M} \sum_{l=L}^{L+M} \langle S,e_l\rangle^2 + \frac{2}{M}\sum_{l=L}^{L+M}\langle S,e_l\rangle z_l+ \frac{1}{M}\sum_{l=L}^{L+M} z_l^2- \sigma^2 = I+II+III,
\end{align}
where $z_L,\ldots,z_{L+M}$ are i.i.d. standard normal random variables also independent from $S$.
 
We now take the expectation
\begin{align*}
\E\left[\hat\sigma^2\right]  - \sigma^2=  \frac{1}{M} \sum_{l=L}^{L+M}\sum_{j=1}^r \lambda_j \langle  \varphi_j, e_l\rangle^2.
\end{align*}
Note that $\langle  \varphi_j, e_l\rangle^2 \leq \|\varphi_j - \varphi_j^{(l-1)}\|_2^2$. In view of Assumption \ref{bias-cond}, we get
\begin{align*}
\frac{1}{M} \sum_{l=L}^{L+M}\sum_{j=1}^r \lambda_j \langle  \varphi_j, e_l\rangle^2 &\leq  r \lambda_{\max}  \frac{c_*^2}{M}\sum_{l=L-1}^{L+M-1} l^{-2s}\lesssim  c_*^2 r \lambda_{\max} L^{-2s}.
\end{align*}
The bound in probability follows easily from the representation (\ref{sigma-est}). Indeed, the second term can be treated using standard deviations bounds for Gaussian combined with a conditioning argument. The third term can be treated with a standard deviation inequality for chi-square distributions. The first term can be treated using (\ref{HW}) again. More specifically, set $\xi = (\xi_1,\ldots,\xi_r)^\top$ and $A = (a_{j,j'})_{1\leq j,j'\leq r}$ with
$$
a_{j,j'} = \frac{\sqrt{\lambda_j \lambda_{j'}}}{M}\sum_{l=L}^{L+M}  \langle \varphi_j, e_l \rangle \langle \varphi_{j'}, e_l\rangle.
$$
Then, we have
$$
\frac{1}{M} \sum_{l=L}^{L+M} \langle S,e_l\rangle^2  - \E\left[\frac{1}{M} \sum_{l=L}^{L+M} \langle S,e_l\rangle^2  \right]= \xi^\top A \xi - \E[\xi^\top A \xi],
$$
with $\|A\|_F\lesssim c_{*}^2 r\lambda_{\max} L^{-2s}$ and $\|A\|_{\infty} \lesssim c_*^2 \sqrt{r } \lambda_{\max} L^{-2s}$.

An union bound argument gives the result. Details of the proof are omitted here.

\end{proof}

\section{Minimax lower bound}\label{lower bound}

In this section, we show that the upper bounds of Corollary~\ref{cor:finite} and Theorem~\ref{th:nonpar} cannot be improved in a minimax sense.\begin{theorem}\label{th:finite:lower}
Let $1\le r< \infty$ and let $\l_{\max}>0$ be a given constant.
Then there exist absolute constants $c_0>0$ and $0<c_1<1$ such that, for any integers $n$ and $l$  satisfying $l\geq 2$, $n\geq l$, we have
$$
\inf_{\hat K_n}\sup_{K \in  {\mathcal K}_{r,l}(\l_{\max}) } \mathbb P \left(   \|\hat K_n - K\|_2^2 \geq c_0 [\lambda_{\max} \wedge \sigma^2]^2\frac{(r\wedge l)l}{n} \right) >c_1
$$
where $\inf_{\hat K_n}$ denotes the infimum over all estimators of $K$.
\end{theorem}

\begin{proof}
Let first $r\le l/2$.
Consider the vector-functions $e(t)=(e_1(t),\dots, e_l(t))$ and $\f (t)=(\f_1(t),\dots, \f_r(t))$ and a subset of ${\mathcal K}_{r,l}(\l_{\max})$ composed of kernels $K$ satisfying \eqref{K} with $\l_j\equiv \g$ and  
$$
\f(t)= H e(t)
$$
for suitable $\g>0$ and suitable $r\times l$ matrices $H$. Orthonormality of functions $\f_j$ implies that $H$ must satisfy $HH^\top=I_r$ where $I_r$ is the $r\times r$ identity matrix, i.e., the rows of $H$ should be orthonormal. To each such matrix $H$ we associate a linear subspace $U_H$ of $\R^l$, which is the linear span of the $r$ rows of $H$. Clearly, ${\rm dim} (U_H)=r$ and $H^\top H$ is the orthogonal projector onto $U_H$ in $\R^l$.

Note that the set of all such spaces $U_H$ is the Grassmannian manifold $G_{r}(\mathbb R^l)$, i.e., the set of $r$-dimensional linear subspaces of $\mathbb R^l$. The Grassmannian manifold $G_{r}(\mathbb R^l)$ is a smooth manifold of dimension $d= r(l-r)$. A natural metric $d(\cdot,\cdot)$ on $G_{r}(\mathbb R^l)$ is defined as follows: for $U,\bar U \in G_{r}(\mathbb R^l)$,
$$
d(U,\bar U) \triangleq \|P_U-P_{\bar U}\|_F = \|H^\top H-\bar H^\top \bar H\|_F
$$ 
where $P_U$ is the orthogonal projector onto $U$ and $H$, $\bar H$ are the $r\times l$ matrices with orthonormal rows associated to $U$ and $\bar U$ respectively. We refer to \cite{Pertti} and \cite{MilnorStasheff} for more details on the Grassmannian manifold. 

From now on, we will identify $U\in G_{r}(\mathbb R^l)$ with the associated orthogonal projector $P_U=H^\top H$.
The behavior of entropy numbers of the Grassmannian manifold is well studied (\cite{Szarek1982}, see also Proposition 8 in \cite{Pajor1998}). In particular, for any $\epsilon \in (0,1)$ there exists a family of orthogonal projectors $\mathcal U \subset G_{r}(\mathbb R^l)$ such that
\begin{align}\label{Grassman-entropy}
|\mathcal U | \geq  \left\lfloor \frac{\bar c}{\epsilon} \right\rfloor^{d}\quad\text{and}\quad   \bar c \epsilon \sqrt{r} \le \|P-Q\|_F \le \frac1{\bar c} \epsilon \sqrt{r}, \; \forall P,Q \in \mathcal{U}, \; P \neq Q,
\end{align}
for some small enough universal  constant $\bar c>0$. Here $|\mathcal U |$ denotes the cardinality of $\mathcal U$. We take in what follows $\epsilon = 1/2$. Set $N = |\mathcal U |$ and $
\mathcal U  = \left\lbrace  P_1,\dots,P_N \right\rbrace$. The associated $H$-matrices will be denoted by $H_1,\dots,H_N$.  Let  $K_j$ be a kernel of the form  \eqref{K} with eigenvalues $\l_i\equiv \g, i=1,\dots, r,$ and  
$$
\f(t)= H_j e(t), \quad j=1,\dots, N,
$$
where $\gamma =a(\sigma^2\wedge \lambda_{\max})\sqrt{\frac{l}{n}}$ and $a \in (0,1)$ is an absolute constant to be chosen later. Consider the set $\cK'=\{K_1,\dots, K_N\}$. Clearly, we have $\cK' \subset  {\mathcal K}_{r,l}(\l_{\max})$.

We now evaluate the Kullback-Leibler divergence between two probability measures induced by the observations $\{X_1(t),\dots, X_n(t), t\in [0,1]\}$ corresponding to the kernels $K_1$ and $K_j$ (with $j\ne 1$). Using the Girsanov formula and the fact that $K_j$ is bilinear in $\{e_k\}$ it is easy to check that this divergence is equal to 
 the Kullback-Leibler divergence between  the $n$-product distributions of the associated Gaussian vectors $ \left(    \int_0^1 e_1(t)d X (t),\dots, \int_0^1 e_l(t)d X (t)\right)$. If $K=K_j$ this vector is distributed as ${\mathcal N}\left(0,\Sigma_j\right)$ with $\Sigma_j = \sigma^2 I_{l} + 
\gamma P_j = (\sigma^2+ \gamma)P_j + \sigma^2 P_j^{\perp}$ and $P_j^{\perp} = I_{ l} - P_j$. Denote the corresponding Gaussian measure by $\PP_j$ and by $\PP_j^{\otimes n}$ its $n$-product. Let ${\rm KL}(\PP,\QQ)$ be the Kullback-Leibler divergence between  two probability measures $\PP$ and $\QQ$.

It is easy to see that all matrices $\Sigma_j$ have the same eigenvalues. Thus,  for any $2\leq j \leq N$ we have
\begin{align*}
{\rm KL}(\PP_1^{\otimes n}, \PP_j^{\otimes n}) &= n \,{\rm KL}(\PP_1, \PP_j) \notag\\
&= \frac{n}{2} \left[ \mathrm{tr}(\Sigma_1^{-1} \Sigma_j) - l- \log \left(  \mathrm{det}(\Sigma_1^{-1}\Sigma_j)  \right) \right]\notag\\
&= \frac{n}{2} \left[ \mathrm{tr}(\Sigma_1^{-1}(\Sigma_j-\Sigma_1)\right].
\end{align*}
Now, $\Sigma_1^{-1} = \frac{1}{\sigma^2 + \gamma}P_1 + \frac{1}{\sigma^2}P_1^{\perp}$, which yields
\begin{align*}
\mathrm{tr}(\Sigma_1^{-1}(\Sigma_j-\Sigma_1)) &= \frac{\gamma}{\sigma^2 + \gamma}\mathrm{tr}(P_1 (P_j - P_1))+ \frac{\gamma}{\sigma^2}\mathrm{tr}(P_1^{\perp}(P_j-P_1))\\
&= \left(\frac{\gamma}{\sigma^2 + \gamma} - \frac{\gamma}{\sigma^2}\right) \left( \mathrm{tr}(P_1P_j) - r\right)\\
& = \frac{\gamma^2}{2(\sigma^2 + \gamma)\sigma^2}\|P_1- P_j\|_F^2\\
&\leq \frac{r\gamma^2}{8 \bar c^2 (\sigma^2 + \gamma)\sigma^2}
\end{align*}
where for the last inequality we have used \eqref{Grassman-entropy} with $\epsilon=1/2$, and the fact that $\mathrm{tr}(P_1P_j) = r - \|P_1 - P_j\|_F^2/2$.
Combining the last two displays, we find
\begin{align*}
{\rm KL}(\PP_1^{\otimes n}, \PP_j^{\otimes n}) &\leq a^2\frac{(\lambda_{\max} \wedge \sigma^2)^2}{8\bar c^2(\sigma^2 + \gamma)\sigma^2}rl ,\quad \forall \ 2\leq j \leq N.
\end{align*}

Recall that we assume $r\le l/2$, so that the dimension of the Grassmannian satisfies $d =r(l-r) \geq rl/2$. Consequently, in view of (\ref{Grassman-entropy}), we have that $\log |\mathcal U|\ge {\tilde c}rl $ for some absolute constant ${\tilde c}>0$. Thus, we get
\begin{align*}
{\rm KL}(\PP_1^{\otimes n}, \PP_j^{\otimes n}) &\leq \frac{1}{16}\log |\mathcal U|,\quad \forall \ 2\leq j \leq N,
\end{align*}
provided $a>0$ is taken sufficiently small independently of $r,l,n,\sigma,\lambda_{\max}$.

Next, for any $1\leq i,j\leq N$ with $i\neq j$,
$$
\| K_i  - K_j \|_2^2 = \gamma^2 \|H_i^\top H_i-H_j^\top H_j\|_F^2=\gamma^2 \|P_i - P_j\|_F^2  \geq c a^2 [\sigma^4\wedge \lambda_{\max}^2]\frac{ r l }{n},
$$
where $c>0$ is a absolute constant and the last inequality is due to~(\ref{Grassman-entropy}).
The result now follows from the last two displays by application of Theorem 2.5 in \cite{Tsybakov-book}. 

Finally, consider the case $r> l/2 $. Note that the classes  ${\mathcal K}_{r,l}(\l_{\max})$ are nested in $r$. Assuming w.l.o.g. that $l$ is even, we get that the minimax risk over ${\mathcal K}_{r,l}(\l_{\max})$ is bounded from below by the  minimax risk on ${\mathcal K}_{l/2,l}(\l_{\max})$.
But the minimax risk on ${\mathcal K}_{l/2,l}(\l_{\max})$ has been already treated above and we have proved that the lower rate is of the order $l^2/n$, which is the desired rate when $r>l/2$.
\end{proof}

\begin{remark}
It is possible to prove a minimax lower bound ensuring that the bound in Theorem \ref{th:smooth} is optimal at least regarding the $n$ dependence. Indeed, by a similar argument to that used in the proof of Theorem \ref{th:finite:lower}, we can prove the existence of an absolute constant $0<c_2<1$ and a constant $c_3>0$ possibly depending on $\sigma^2,\lambda_{\max},\rho,r$ such that, for any integer $n\geq 1$ we have
$$
\inf_{\hat K_n}\sup_{K \in  \overline{\cK}_r(s,\rho;\lambda_{\max})} \mathbb P \left(   \|\hat K_n - K\|_2^2 \geq c_3  
 \min\left(n^{-\frac{2s}{2s+1}},  n^{-s/(s+1)}\right)\right) >c_2
$$
where $\inf_{\hat K_n}$ denotes the infimum over all estimators of $K$. Specifying the dependence of the minimax rate on parameters $\sigma^2,\lambda_{\max},\rho,r$ remains an interesting open question.

\end{remark}
%
%%

%%
%
%
%
%
%
%%
%
%%
%%
%
%
%
%
%
%
%
%
%
%
%
%
%%

%
%
%
%
%
%%%%%%%%%%%%%%%%
%
%%% The Appendices part is started with the command \appendix;
%%% appendix sections are then done as normal sections
%%% \appendix
%
%%% \section{}
%%% \label{}
%
%%% If you have bibdatabase file and want bibtex to generate the
%%% bibitems, please use
%%%
%%%  \bibliographystyle{elsarticle-harv} 
%%%  \bibliography{<your bibdatabase>}
%
%%% else use the following coding to input the bibitems directly in the
%%% TeX file.

%\textcolor{red}{Do we need title "References" for the references section?}

\end{document}